\documentclass{article}

\usepackage[utf8]{inputenc} 
\usepackage{url}

\usepackage{amsmath,amsfonts,amssymb}
\usepackage{amsthm}
\numberwithin{equation}{section}
\usepackage{graphicx}
\usepackage{enumerate}
\usepackage{multirow,bigdelim}
\usepackage{mathtools}
\usepackage{soul} 
\usepackage{pstricks}
\usepackage[ruled, boxed,  linesnumbered, norelsize]{algorithm2e}
\usepackage[margin=3cm]{geometry}
\usepackage{array}
\usepackage[normalem]{ulem}
\usepackage{url}
\usepackage{hyperref}
\usepackage{enumerate}
\usepackage{cases}
\usepackage{mathtools}
\usepackage{todonotes}

\hypersetup{
	colorlinks,
	linkcolor={red!50!black},
	citecolor={blue!50!black},
	urlcolor={blue!80!black}
}

\newcommand{\inProd}[2]{\langle #1 , #2 \rangle }
	
\newcommand{\SOC}[2]{{\mathcal{L}^{#2} _{#1}}}

\newcommand{\spanVec}{\operatorname{span}}
\newcommand{\norm}[1]{\lVert{#1}\rVert}
\newcommand{\innprod}[2]{\langle #1 , #2 \rangle }

\newcommand{\stdCone}{ {\mathcal{K}}}

\newcommand{\stdSpec}{ S}

\newcommand{\matRank}{{\mathrm{ rank } \,}}	
\newcommand{\matRange}{{\mathrm{ range } \,}}

\renewcommand{\Re}{\mathbb{R}}

\renewcommand{\S}{\mathcal{S}}                    
\newcommand{\tr}{\mathrm{tr}}

\newcommand{\jAlg}{\mathcal{E}}

\newcommand{\rank}{\operatorname{rank}}

\newcommand{\proj}{\operatorname{proj}}
\newcommand{\dist}{\operatorname{dist}}

\newcommand{\orR}[1]{\Re^{#1}_{\downarrow}}

\newcommand{\ep}{\varepsilon}
\newcommand{\Diag}{\operatorname{Diag}}
\newcommand{\diag}{\operatorname{diag}}
\newcommand{\dom}{\operatorname{dom}}

\theoremstyle{definition}

\newtheorem{definition}{Definition}[section]

\newtheorem{example}[definition]{Example}

\theoremstyle{theorem}

\newtheorem{proposition}[definition]{Proposition}

\newtheorem{corollary}[definition]{Corollary}

\newtheorem{theorem}[definition]{Theorem}
\newtheorem*{proposition*}{Proposition}

\theoremstyle{remark}
\newtheorem{remark}[definition]{Remark}

\title{Eigenvalue programming beyond matrices}
\author{
	Masaru Ito%
		\thanks{Department of Mathematics, College of Science and Technology, Nihon University,
			1-8-14 Kanda-Surugadai, Chiyoda-Ku, Tokyo 101-8308, Japan.
			This author
			was supported partly by the JSPS Grant-in-Aid for Early-Career Scientists 	21K17711.
			Email: (\texttt{ito.masaru@nihon-u.ac.jp)}.}
	\and
	Bruno F. Louren\c{c}o%
\thanks{Department of Fundamental Statistical Mathematics, Institute of Statistical Mathematics, Japan.
This author
	was supported partly by the JSPS Grant-in-Aid for Early-Career Scientists 	23K16844. Email:~(\texttt{bruno@ism.ac.jp}).}
}

\begin{document}
\maketitle
\begin{abstract}
	In this paper we analyze and solve eigenvalue programs, which  consist of the task of  minimizing a function subject to constraints on the ``eigenvalues'' of the decision variable.
	Here, by making use of the FTvN systems framework introduced by Gowda,  we interpret ``eigenvalues'' in a broad fashion going beyond the usual eigenvalues of matrices. 
	This allows us to shed new light on classical problems such as inverse eigenvalue problems and also leads to new applications.
	In particular, after analyzing and developing a simple projected gradient algorithm for general eigenvalue programs, 
	we show that eigenvalue programs can be used to express what we call  \emph{vanishing quadratic constraints}. A {vanishing quadratic constraint} requires that a given system of convex quadratic inequalities be satisfied and at least a certain number of those inequalities must be tight.
	As a particular case, this includes the problem of  finding a point $x$ in the intersection of $m$ ellipsoids in such a way that $x$ is also in the boundary of at least $\ell$ of the ellipsoids, for some fixed $\ell > 0$.
	At the end, we also present some numerical experiments.
\end{abstract}
\section{Introduction}
In this paper, we consider problems of the following 
form.
\noindent\begin{align*}
{\min _{x \in \jAlg}} & \quad f(x)
\label{eq:eig_prog}\tag{Eig-Prog}	\\ 
\mbox{subject to} & \quad \lambda(x) \in \mathcal{C},
\end{align*}
where  $f:\jAlg \to \Re$ is a smooth function, 
$\jAlg$ is a finite dimensional Euclidean space, $\mathcal{C} \subseteq \Re^r$ and $(\jAlg,\Re^r,\lambda)$ is a so-called \emph{FTvN system}, which was originally introduced by Gowda \cite{G19}.
Informally, $\lambda:\jAlg \to \Re^r$ has the role of ``eigenvalue map'' and \eqref{eq:eig_prog} corresponds to the task of minimizing $f$ subject to the constraint that the ``eigenvalues'' of $x$ must belong to a certain set $\mathcal{C}$.
Because of this, we will call \eqref{eq:eig_prog} an \emph{eigenvalue program} and this will be the main subject matter of this work.

The motivation  for this work is to provide a broad, yet simple framework, for solving optimization problems which requires constraints on the eigenvalues of $x$. 
In particular, we will cover the case where $\jAlg$ is $\S^n$ (the space of $n\times n$ real symmetric matrices) and $\lambda$ is the usual matrix eigenvalue map, where the eigenvalues are ordered from largest to smallest. 
This includes the particular case of rank constraints, since we may force an $n\times n$ symmetric positive semidefinite matrix to have rank less or equal than $r$ by requiring that the $n-r$ smallest eigenvalues vanish.
And, slightly more generally, we may constrain the rank of a general matrix by forcing the smaller singular values to vanish.
In this work, however, we present techniques to control the range of the whole eigenvalue map, not only the smallest/largest eigenvalues.

In addition, as we will discuss in Section~\ref{sec:ex}, there are also other interesting problems that can be modelled as an {eigenvalue program} beyond matrix problems. 
For example, suppose that we have $m$ ellipsoids and we would like to find an intersection point that is at the boundary of at least $\ell$ of those ellipsoids, for some $\ell \leq m$. Surprisingly, such a problem can also be cast as an eigenvalue program in an appropriate setting without making use of the usual matrix eigenvalues.
More generally, given $m$ quadratic inequalities, the problem of finding points that satisfy at least $\ell$ of those inequalities exactly, can also be cast as an eigenvalue program, as we will discuss in Section~\ref{sec:ex}.

The problem in \eqref{eq:eig_prog} at first glance seems highly nonconvex, because, as exemplified previously, we may use it to express  rank constraints.  
For the sake of concreteness, 
let $P$ be a polyhedral set in $\Re^r$,
let $\jAlg = \S^r$ be the space of real $r\times r$ symmetric matrices and $\lambda$ the usual eigenvalue map. If we set $\mathcal{C} = (\Re_+^{s} \times \{0\}^{r-s})\cap P$, the feasible set of \eqref{eq:eig_prog} consists of positive semidefinite matrices with rank less or equal than  $s$  such that their eigenvalues are in $P$.
Clearly, $\lambda^{-1}(\mathcal{C})$ (i.e., the feasible set of \eqref{eq:eig_prog}) will be nonconvex in general. Now, \emph{nonconvex} does not mean \emph{hard}, but it often means that we may need to get more creative in order to better understand the problem.

In \cite{G19}, Gowda made several fundamental contributions and he showed that linear optimization over ``$\lambda^{-1}(\mathcal{C})$'' can be essentially converted to a linear optimization problem over $\mathcal{C}$ and, similarly, in order to project onto 
 ``$\lambda^{-1}(\mathcal{C})$'', it is enough to know how to project a vector $v \in \Re^r$ onto $\mathcal{C}$. 
This is notable, because in the example discussed just now, linear optimization over $\mathcal{C}$ would correspond to a \emph{linear program}, in stark contrast to linear optimization over $\lambda^{-1}(\mathcal{C})$.
 
With these points in mind, our motivation for this work is to leverage the theoretical findings of \cite{G19} into a simple yet practical algorithm and showcase some novel modelling applications of eigenvalue programs.
Our contributions are as follows.
\begin{itemize}
	\item We use the FTvN system framework to analyze and implement a simple projected gradient algorithm for solving \eqref{eq:eig_prog}, see Section~\ref{sec:projg}.
	\item We discuss some of the modelling capabilities of \eqref{eq:eig_prog}, see Section~\ref{sec:ex}. We start by revisiting \emph{inverse eigenvalue problem} but with general eigenvalue constraints.
	Next, we show that when $\jAlg$ corresponds to a Jordan algebra associated to a direct product of second-order cones, we can solve feasibility problems with the so-called \emph{vanishing quadratic constraints}, which require that a certain number of convex quadratic inequalities be satisfied with equality. 
	In particular, given $m$ ellipsoids, we can model and solve the problem of finding an intersection point that is in the boundary of at least $\ell$ of those ellipsoids using {vanishing quadratic constraints}, see Section~\ref{sec:ellipsoids}.
\end{itemize}
This work is divided as follows. 
In Section~\ref{sec:prel}, we set up the notation and review notions from variational analysis. In Section~\ref{sec:ftvn}, we present a self-contained account of the theory of FTvN systems and discuss some fundamental examples. 
In Section~\ref{sec:projg} we discuss and analyze a projected gradient algorithm for solving \eqref{eq:eig_prog}. In Section~\ref{sec:ex} we present some modelling examples together with numerical experiments. Finally, 
in Section~\ref{sec:conc} we conclude this paper.

\section{Preliminaries}\label{sec:prel}
Let $\jAlg$ be a finite dimensional real vector space equipped with an inner product $\langle\cdot,\cdot\rangle$.
We denote by $\|\cdot\|=\sqrt{\innprod{\cdot}{\cdot}}$ the induced norm on $\jAlg$.
We also write $\norm{w}_2=\sqrt{w^Tw}$ for $w \in \Re^r$.
For a set $S \subset \jAlg$, we define
\[
\dist(x,S) \coloneqq \inf_{z \in S}\norm{x-z},\quad
\proj_S(x) \coloneqq \{z \in S \mid \norm{z-x} = \dist(x,S)\}.
\]
We denote the indicator function of $S$ by $\delta _S$.
For a vector $x \in \Re^r$, we denote by $x^\downarrow \in \Re^r$ the vector that corresponds to ordering the elements of $x$ from largest to smallest, so that 
$x^\downarrow _1 \geq \cdots \geq x^\downarrow_n$. We also write $S_\downarrow = \{x^\downarrow ~|~ x \in S\}$ for a set $S \subseteq \Re^r$.
 
\paragraph{Some variational analysis}
Before we move forward we recall some notions from variational analysis, for more details see \cite{RW,Mo18}. Let $f : \jAlg \to \Re \cup \{+\infty\}$ be a function. 
We say that $d$ is a \emph{regular subgradient of $f$ at $x$} if
\begin{equation}\label{eq:reg_sub}
\liminf_{\substack{v\to 0 \\ v\neq 0} } \frac{f(x+v) - f(x) - \inProd{d}{v}}{\norm{v}} \geq 0.
\end{equation} 
The set of regular subgradients of $f$ at $x$ is called the \emph{regular subdifferential} of $f$ at $x$ and is denoted by $\hat \partial f(x)$. 
 
We say that $d$ is an \emph{approximate subgradient} (also called \emph{limiting subgradient}) of $f$ at $x$ if
there are sequences $\{x^k\}$, $\{d^k\} \subset \jAlg$ such that $d^k \in \hat \partial f(x^k)$ holds for every $k$ and the following limits hold:
\[
x^k \to x, \qquad f(x^k)\to f(x), \qquad d^k \to d.
\]
The set of approximate subgradients of $f$ at $x$ is denoted by 
$\partial f(x)$ and is called the \emph{limiting subdifferential of $f$ at $x$}.  We also define
$\dom \partial f(x) = \{x \in \jAlg \mid \partial f(x) \neq \emptyset\}$.
If $x$ is such that $0 \in \partial f(x)$, then $x$ is said to be a \emph{critical point}.
 
Considering different subdifferentials of the indicator function $\delta_S$ of a given subset $S \subseteq \jAlg$ leads to different notions of normal cones. In particular, the regular and the limiting normal cones of a set $S \subseteq \jAlg$ at some $x \in S$ are given by $\hat N _S(x) \coloneqq \hat \partial \delta_{S}(x)$ and $N _S(x) \coloneqq \partial \delta_{S}(x)$, respectively. For more details, 
see \cite[Chapters~6, 8 and Exercise 8.14]{RW}.

\section{FTvN systems}\label{sec:ftvn}
The notion of a Fan-Theobald-von Neumann system introduced by Gowda is a broad generalization of several notions developed throughout the literature to handle objects that behave similarly to the usual matrix eigenvalue map \cite{G19,JJ23}. 
%
\begin{definition}\label{def:FTvN}
For a map $\lambda:\jAlg\to\Re^r$, we say that $(\jAlg, \Re^r, \lambda)$ is a \emph{Fan-Theobald-von Neumann (FTvN) system} if
\begin{enumerate}[$(i)$]
\item $\|\lambda(x)\|_2=\|x\|$, $\forall x \in \jAlg$
\item $\inProd{x}{y} \leq \inProd{\lambda(x)}{\lambda(y)}$, $\forall x,y \in \jAlg$
\item For any $c \in \jAlg$ and $\mu \in \lambda(\jAlg)$, the following set is nonempty:
\begin{equation}\label{eq:FTvN-U}
U(c,\mu) := \{z \in \jAlg \mid \lambda(z)=\mu,~ \inProd{c}{z} = \inProd{\lambda(c)}{\lambda(z)}\}.
\end{equation}
\end{enumerate}
\end{definition}
\begin{remark}
The definition given in \cite{G19,JJ23} is more general than Definition~\ref{def:FTvN}, because it is considered over arbitrary real inner product spaces, while here we restrict ourselves to the finite dimensional case where one of the spaces is $\Re^r$. Also, item~$(iii)$ is stated in a slightly different manner.
\end{remark}
The first two items in Definition~\ref{def:FTvN} are relatively natural: they require that the eigenvalue map $\lambda$ be compatible with the norm used in $\jAlg$ and that the inner product satisfy an inequality analogous to a classical inequality over symmetric matrices (e.g., \cite[Theorem~2.2]{Le96}).
The third item is somewhat more technical, but can be seen as complementing item $(ii)$ as follows: if $c, y \in \jAlg$ are fixed, item~$(ii)$ implies that $\inProd{c}{y} \leq \inProd{\lambda(c)}{\lambda(y)}$. 
However, 
applying item~$(iii)$ to $c$ and $\lambda(y)$, we get that this inequality becomes tight for at least one $z \in \jAlg$ that has the same eigenvalues as $y$. 

For FTvN systems, the projected gradient algorithm discussed in this paper relies on evaluating the map $\lambda(\cdot)$ and being able to obtain an element of the set $U(c,\mu)$.
For many cases, this is done by ``spectral decompositions'' as we will see later.

We define the \emph{rank} of $x \in \jAlg$ by
\begin{equation}\label{eq:rank}
\rank(x) \coloneqq \text{ number of nonzero elements in } \lambda(x). 
\end{equation}
For a FTvN system $(\jAlg, \Re^r, \lambda)$, a \emph{spectral set} in $\jAlg$ is a subset $\stdSpec$ of $\jAlg$ of the form
\[
\stdSpec = \lambda^{-1}(Q) \text{ for some } Q \subseteq \Re ^r.
\]
Observe that any spectral set $\stdSpec = \lambda^{-1}(Q)$ satisfies
\begin{equation}\label{eq:spec-prop}
\lambda(\stdSpec) = \lambda(\jAlg) \cap Q,
\quad
\stdSpec = \lambda^{-1}(\lambda(\stdSpec)).
\end{equation}

A very useful property of FTvN systems is that some optimization problems over $\jAlg$ can be equivalently reformulated over $\Re^r$ which may lead to significant simplifications.

\begin{proposition}[Gowda \cite{G19}]\label{prop:G}
Let $(\jAlg,\Re^r,\lambda)$ be a FTvN system and $\stdSpec$ be a spectral set in $\jAlg$. For any real-valued function $\varphi:\Re^r \to \Re$ and $c \in \jAlg$, we have
\begin{equation}\label{eq:spec-sup-reduce}
\sup_{x \in \stdSpec}\{\innprod{c}{x} + \varphi(\lambda(x))\} = \sup_{w \in \lambda(\stdSpec)}\{\innprod{\lambda(c)}{w}+\varphi(w)\}.
\end{equation}
Morevoer, $x^* \in \stdSpec$ attains the supremum of the left-hand side if and only if $\innprod{c}{x^*}=\innprod{\lambda(c)}{\lambda(x^*)}$ and $\lambda(x^*)$ attains the supremum of the right-hand side. 
In other words, this is if and only if $x^* \in U(c,w^*)$ for some $w^* \in \lambda(S)$ that attains the supremum of the right hand side.
\end{proposition}

\begin{proof}
The assertion is essentially \cite[Corollary~3.3]{G19} which follows from more general results therein. 
Here, we give a direct proof for the sake of self-containment. 
Let $L$ and $R$ denote the left and the right hand sides of \eqref{eq:spec-sup-reduce}, respectively.

Let $x \in S$. By item $(ii)$ of Definition~\ref{def:FTvN}, $w \coloneqq \lambda(x)$ is a feasible solution to right-hand side of \eqref{eq:spec-sup-reduce} satisfying 
$\inProd{c}{x} \leq \inProd{\lambda(c)}{w}$. Therefore, $L \leq R$. 

Conversely, for any $w \in \lambda(\stdSpec)$,  the set $U(c,w)$ is nonempty, by item $(iii)$ of Definition~\ref{def:FTvN}. 
With that we pick \emph{any} $x \in U(c,w)$ and, by definition, we have $\lambda(x)=w$ and $\innprod{c}{x}=\innprod{\lambda(c)}{\lambda(x)} = \innprod{\lambda(c)}{w}$. 
These equalities imply that
\begin{equation*}
\innprod{c}{x} + \varphi(\lambda(x))
= \innprod{\lambda(c)}{w} + \varphi(w).
\end{equation*}
We also have $x \in S$, since 
$x \in \lambda^{-1}(w)\subseteq \lambda^{-1}(\lambda(\stdSpec))=\stdSpec$ (see \eqref{eq:spec-prop}).
In conclusion, for the two maximization problems associated to  $L$ and $R$, we can recover a feasible solution of the former from a feasible solution of the latter, without changing the objective value. This indicates $L \geq R$ which shows that \eqref{eq:spec-sup-reduce} holds.

The development done so far implies 
that if $x^* \in S$ is an optimal solution to the left-hand side of \eqref{eq:spec-sup-reduce}, then 
$w \coloneqq \lambda(x^*)$ must be an optimal solution to the right-hand side. In view of $L = R$, we have
$\innprod{c}{x^*} + \varphi(\lambda(x^*))
= \innprod{\lambda(c)}{\lambda(x^*)} + \varphi(\lambda(x^*))$, 
which leads to $\innprod{c}{x^*}=\innprod{\lambda(c)}{\lambda(x^*)}$. 
Conversely, suppose that $x^* \in S$ is such that $\innprod{c}{x^*}=\innprod{\lambda(c)}{\lambda(x^*)}$ and $\lambda(x^*)$ attains the supremum of the right-hand side of \eqref{eq:spec-sup-reduce}.
Again, in view of $L = R$, $x^*$ must be optimal to the left-hand side of \eqref{eq:spec-sup-reduce}.
\end{proof}

For instance, when $\varphi(\cdot)\equiv 0$, this assertion says that linear optimization $\sup_{x\in \stdSpec}\innprod{c}{x}$ for a spectral set $S$ can be reduced to $\sup_{w \in \lambda(\stdSpec)}\innprod{\lambda(c)}{w}$ whose solution $w^*$ can be used to recover a solution to $\sup_{x \in \stdSpec}\innprod{c}{x}$ through the computation of an element $x^*\in U(c,w^*)$.

Taking $\varphi(x)=-\norm{x}^2/2$, we obtain the following useful fact regarding projections, see also \cite[Corollary~3.8]{G19}. 

\begin{corollary}\label{cor:proj-spec}
Let $(\jAlg,\Re^r,\lambda)$ be a FTvN system and $\stdSpec$ be a spectral set in $\jAlg$. For any $z \in \jAlg$, we have
$$
\proj_{\stdSpec}(z) = \bigcup_{{ w^* \in \proj_{\lambda(\stdSpec)}}(\lambda(z))}  U(z,w^*).
$$
\end{corollary}
\begin{proof}
Since $\norm{x-z}^2/2 = \norm{x}^2/2 - \innprod{z}{x} + \norm{z}^2/2$ holds, $\proj_{\stdSpec}(z)$ is the set of solutions of the problem of maximizing
\[
\innprod{z}{x}-\frac{1}{2}\norm{x}^2 = \innprod{z}{x}-\frac{1}{2}\norm{\lambda(x)}_2^2,
\]
over $x \in \stdSpec$,
where the equality follows from item $(i)$ of Definition~\ref{def:FTvN}.
Therefore, the result follows by taking $\varphi(x)\coloneqq-\norm{x}^2/2$ in Proposition~\ref{prop:G}.
\end{proof}
We mention in passing that in the context of spectral sets over symmetric matrices there is a related earlier result by Lewis and Malick \cite[Theorem~A.1]{LM08}.
\subsection{Examples of FTvN systems}\label{ssec:ex}
We give some examples of FTvN systems. 
In the examples discussed here, we will always check that 
the eigenvalue mapping $\lambda$ is 
\emph{semialgebraic}, which will 
significantly simplify the analysis of our projected gradient method. We recall that a {semialgebraic set} in $\Re^n$ a finite union of sets of the form
$$\{x \in \Re^n ~|~ p_{i}(x)=0,\, i=1,\ldots,k \text{ and } q_j(x)<0,\,j=1,\ldots,\ell\},$$
where $p_i$ and $q_j$ are polynomials of real coefficients. A \emph{semialgebraic function} is a map $f:\Re^n \to \Re^m$ whose graph $\{(x,y)~|~y=f(x)\}$ is a semialgebraic set in $\Re^{n+m}$.
Basic facts on semialgebraic sets and functions can be seen in \cite{Boc98}, although we only make use of the most elementary properties.

We now start with what is perhaps the most basic example of FTvN system.
\begin{example}[The symmetric matrices]\label{ex:sym}
The vector space $\S^n$ of $n$ by $n$ real symmetric matrices can be seen as a FTvN system $(\S^n,\Re^n,\lambda)$ with
the eigenvalue map $\lambda: \S^n \to \orR{n}$ defined as follows: 
\[\lambda(X)\coloneqq (\lambda_1(X),\ldots,\lambda_n(X)),
\]
where $\lambda _{i}(X)$ is the $i$-th largest eigenvalue of $X$.
Here, we assume that $\S^n$ is equipped with the trace inner product, so that 
$\norm{X} = \norm{\lambda(X)}_2$ holds for all $X \in \S^n$. 
Item~$(ii)$ of Definition~\ref{def:FTvN} follows from a classical inequality (e.g., \cite[Theorem~2.2]{Le96}).
Finally, we take a look at item~$(iii)$ of Definition~\ref{def:FTvN}. For $C \in \S^n$ and $\mu \in \Re^r_{\downarrow}$, an element of $U(C,\mu)$ for $C \in \S^n$ and $\mu \in \lambda(\Re^n)=\orR{n}$ can be computed as follows. Letting $C=U^T\Diag(\lambda(C))U$ an eigenvalue decomposition of $C$ with an orthogonal matrix $U \in \Re^{n\times n}$, we obtain
\[
U^T\Diag(\mu)U \in U(C,\mu).
\]
An entirely analogous analysis can also be carried out for complex Hermitian matrices.
That $\lambda$ is a semialgebraic map is well-known, but it also follows from a more general result on Jordan algebras, see Proposition~\ref{prop:semialg}.
\end{example}

\subsubsection{Jordan algebras}\label{sec:jalg}

We discuss an important subclass of FTvN systems, the Euclidean Jordan algebras.  For more details see \cite{FK94,Faybusovich2008}.
A finite-dimensional real inner product space $(\jAlg,\innprod{\cdot}{\cdot})$ is called a \emph{Euclidean Jordan algebra} if it admits an operation $\circ:\jAlg\times \jAlg \to \jAlg$ (called a \emph{Jordan product}) which has an identity element $e \in \jAlg$ (i.e., $x\circ e = x$ for all $x\in \jAlg$) and satisfies
$$
x\circ y = y\circ x, \quad
x\circ (x^2\circ y) = x^2 \circ (x \circ y),\quad
\innprod{z\circ x}{y} = \innprod{z}{x\circ y},
$$
for all $x,y,z \in \jAlg$, where $x^2 = x\circ x$.

The \emph{cone of squares} associated to $\jAlg$ is given by $\stdCone = \{x \circ x \mid x \in \jAlg\}$ and it is known to be a \emph{symmetric cone}, i.e., a self-dual homogeneous cone. Conversely, every symmetric cone in finite dimensions is known to arise as the cone of squares of some Euclidean Jordan algebra.

An element $v \in \jAlg$ is called an \emph{idempotent} if $v\circ v = v$.
Two idempotents $u,v \in \jAlg$ are \emph{orthogonal} if $u\circ v=0$.
Note that $\innprod{u}{v}=\innprod{u^2}{v}=\innprod{u}{u\circ v}=0$ for orthogonal idempotents $u$ and $v$.

The maximal number of nonzero idempotents which are pairwise orthogonal is well-defined and called the \emph{rank} of $\jAlg$.
A remarkable fact about Euclidean Jordan algebras $\jAlg$ is that they admit a form of spectral decomposition very similar to the matrix case. Namely if $\jAlg$ has rank $r$, then any $x \in \jAlg$ can be expressed as
$$
x = \lambda_1(x) v_1 + \cdots + \lambda_r(x) v_r,
$$
for some $(\lambda_1(x),\ldots,\lambda_r(x)) \in \orR{r}$ and the $v_1,\ldots,v_r \in \jAlg$ are nonzero such that $v_i^2=v_i,~v_i \circ v_j=0~(i\ne j),~\sum_{i=1}^r v_i = e$.
Here, $\lambda_1(x),\ldots,\lambda_r(x)$ are called the \emph{eigenvalues} of $x$ and $\{v_1,\ldots,v_r\}$ is called a \emph{Jordan frame of $x$}. 
The eigenvalues of $x$ are unique.

The trace map $\tr(x):=\lambda_1(x)+\cdots+\lambda_r(x)$ yields an inner product $(x,y)\mapsto \tr(x \circ y)$ on $\jAlg$.
With that, we can define the map 
$\lambda: \jAlg \to \Re^r_{\downarrow}$ that maps an element to its eigenvalues ordered from largest to smallest. 

Equipped with the trace inner product on $\jAlg$, the tuple $(\jAlg,\Re^r,\lambda)$ forms a FTvN system, see \cite[Section~4]{G19}. 
For the sake of completeness, we check the details. 
First, item $(i)$ of Definition~\ref{def:FTvN}
follows by definition of the trace inner product.
Then, item $(ii)$ follows from \cite[Theorem 23]{baes07} which states that $\inProd{x}{y} \leq \innprod{\lambda(x)}{\lambda(y)}$. Moreover, given $c \in \jAlg$ and $\mu \in \lambda(\jAlg)=\Re^r_{\downarrow}$, considering the spectral decomposition $c = \lambda_1(c)v_1+\cdots+\lambda_r(c)v_r$,
we construct the following element of $U(c,\mu)$:
\begin{equation}\label{eq:ucmu}
\mu_1v_1+\cdots+\mu_r v_r \in U(c,\mu).
\end{equation}
We also observe that the rank of $x \in \jAlg$ as defined in \eqref{eq:rank} corresponds to the number of nonzero eigenvalues that $x$ has.

Finally, let us check that the map  $\lambda$ is semialgebraic, which may be obvious for those familiar with Jordan algebras. However, since this fact is not  clearly stated in the classical references, we provide a proof. 
\begin{proposition}\label{prop:semialg}
In a finite-dimensional Euclidean Jordan algebra $\jAlg$ of rank $r$, the map $\lambda: \jAlg \to \Re^r_{\downarrow}$ is semialgebraic.
\end{proposition}
\begin{proof}
 Consider the function $\det:\jAlg \to \Re $ such that $\det(x) = \prod_{i=1}^{r} \lambda_i(x)$ holds. It is
known that the composition of an elementary symmetric polynomial with $\lambda$ leads to a polynomial in $x$, see \cite[Theorem~III.1.2]{FK94}.
In particular, since  $\det$ is the composition of $\lambda$ with the elementary symmetric polynomial $\prod_{i=1}^r t_i$, we conclude that $\det$ is a polynomial function. 
Furthermore, since $\lambda_{i}(x-te) = \lambda_i(x) -t$ holds for all $x$ and $i = 1,\ldots, r$, 
the eigenvalues of a given $x \in \jAlg$ are exactly 
the roots of the polynomial $t \mapsto \det(x-te)$. 
With this, the graph of $\lambda$ 
can be written as 
the semialgebraic set
\[\{(x,u) \in \jAlg \times \Re^r_{\downarrow} \mid \forall t \in \Re,\,\, \det(x-te) = (t-u_1)\cdots (t-u_r) \},\] so $\lambda$ is indeed a semialgebraic function\footnote{Here we used the fact that the solution set of a first-order formula quantified over semialgebraic sets and functions must be semialgebraic as well, e.g., \cite[Proposition~2.2.4]{Boc98}.}.
\end{proof}
It is well-known that the $n \times n$ real symmetric matrices as in Example~\ref{ex:sym} form a Jordan algebra with the Jordan product being $X \circ Y = (XY+YX)/2$ and corresponding symmetric cone being the positive semidefinite matrices.
Next, we recall another well-known example.

\begin{example}[The algebra of second-order cones]\label{ex:soc}
		Consider the vector space $\jAlg \coloneqq \Re^{n}\times \Re$ with
		the Jordan product $\circ$ such 
		that for $(x,t),(y,u) \in \Re^{n}\times \Re$ we have
		\begin{equation}\label{eq:jor_soc}
		(x,t) \circ (y,u) \coloneqq \frac{\sqrt{2}}{2}(ux+ty,\inProd{x}{y}+tu),
		\end{equation}
		where $\inProd{\cdot}{\cdot}$ is the usual dot product. By convention we write $\Re^{n+1}$ instead of $\Re^n \times \Re$ and for $(x,t) \in \Re^{n+1}$, we assume that $x \in \Re^n$ and $t \in \Re$.		
		With that, $\Re^{n+1}$ can be seen as a Euclidean Jordan algebra, e.g., see \cite[Chapter~II, Section~1]{FK94}, \cite[Section~2]{FLT02} or \cite[page~141]{Sturm2000}. Depending on the reference, the $\frac{\sqrt{2}}{2}$ factor may be omitted in \eqref{eq:jor_soc}.
		In any case, the corresponding cone of squares is the second-order cone $\SOC{2}{n+1} \coloneqq \{(x,t) \in \Re^n \times \Re \mid \norm{x}_2 \leq t\}$, where $\norm{\cdot}_2$ is the usual Euclidean norm.	
			
	Given $(x,t) \in \Re^{n+1}$, we have the spectral decomposition
	\begin{equation}\label{eq:spect_dec}
	(x,t) = \lambda_+(x,t) e_+ + \lambda_-(x,t) e_-,\quad \lambda_\pm(x,t) = \frac{\sqrt{2}}{2}(t\pm \norm{x}_2),\quad e_\pm = \frac{\sqrt{2}}{2}(\pm w_x, 1),
	\end{equation}
	where $w_x = x/\norm{x}_2$ if $x \neq 0$, otherwise $w_x$ can be taken to be any vector of unit norm.

	The inner product in $\Re^{n+1}$ satisfies 
	$\tr((x,t) \circ (y,u)) = \inProd{x}{y}+tu$, which coincides with the usual Euclidean inner product on $\Re^{n+1}$ and explains the $\frac{\sqrt{2}}{2}$ factor in \eqref{eq:jor_soc}.
	With the eigenvalue map $\lambda(x,t):=(\lambda_+(x,t),\lambda_-(x,t))$, the tuple $(\Re^{n+1},\Re^2,\lambda)$ forms a FTvN system, 
where the inner products in  $\Re^{n+1},\Re^2$ are the usual Euclidean one.

	Following \eqref{eq:ucmu}, given $c=(u,t)$ and $\mu=(\mu_+,\mu_-)\in \orR{2}$, the spectral decomposition $c=\lambda_+(c)e_+ + \lambda_-(c)e_-$ of $c$ allows us to compute an element of $U(c,\mu)$ as follows.
	\begin{equation}\label{eq:ucmu_socp}
	\mu_+ e_+ + \mu_- e_- \in U(c,\mu).
	\end{equation}
\end{example}

\subsubsection{Singular values}
For positive integers $m,n$ and $r=\min(m,n)$, we denote by $\sigma(X)=(\sigma_1(X),\ldots,\sigma_r(X))$ the ordered vector of singular values of $X \in \Re^{m\times n}$ so that $\sigma_1(X) \geq \cdots \geq \sigma_r(X)$. 
We use the inner product \[\inProd{X}{Y} = \tr(X^TY),\] for matrices $X,Y \in \Re^{m \times n}$.
With that, we check that the tuple $(\Re^{m\times n}, \Re^r, \sigma)$ is a FTvN system.
Item~$(i)$ of Definition~\ref{def:FTvN} holds by definition, as for item~$(ii)$, 
it follows from \cite[Theorem~4.6]{lewis05a}.
Finally, to compute an element of $U(C,\mu)$ for $C \in \Re^{m\times n}$ and $\mu \in \sigma(\Re^{m\times n})=(\Re_+^r)_\downarrow$, let $C=U^T\Diag(\sigma(C))V$ be a singular value decomposition of $C$ with square orthogonal matrices $U\in \Re^{n\times n}$ and $V\in \Re^{m\times m}$. Then, we have
\[
U^T \Diag(\mu)V \in U(C,\mu).
\]
Finally, the map $\sigma$ is semialgebraic 
because $\sigma(X)$ correspond to the square-root of the eigenvalues of $XX^T$.

Complex analogues can also be constructed with the inner product $\operatorname{Re}(\tr(X^*Y))$ where $X^*$ denotes the adjoint of $X \in \mathbb{C}^{m \times n}$.

\subsubsection{Direct products}\label{sec:prod}
Other FTvN systems can be constructed by taking the direct product of the above examples.
Given $m$ FTvN systems $(\jAlg^{(i)}, \Re^{r_i}, \lambda^{(i)})$, $i=1,2,\ldots,m$, defining
\begin{equation}\label{eq:prod-FTvN}
\jAlg = \jAlg^{(1)} \times \cdots \times \jAlg^{(m)},\quad  r = r_1+\cdots+r_m,\quad \lambda = (\lambda^{(1)},\ldots,\lambda^{(m)}),
\end{equation}
the tuple $(\jAlg,\Re^r,\lambda)$ forms a FTvN system. Here, the inner product is given 
by \[\inProd{(x_1,\ldots,x_m)}{(y_1,\ldots,y_m)} = \sum _{i=1}^m \inProd{x_i}{y_i}_{i},\] where 
$\inProd{\cdot}{\cdot}_i$ indicates the inner product over each $\jAlg^{(i)}$.
Note that, given $c=(c_1,\ldots,c_m) \in \jAlg$ and $\mu=(\mu_1,\ldots,\mu_m) \in \lambda(\jAlg)$, we have
\begin{equation}\label{eq:ucmu_prod}
U(c,\mu) = U(c_1,\mu_1) \times \cdots \times U(c_m,\mu_m).
\end{equation}
In \eqref{eq:ucmu_prod} the ``$\supseteq$'' inclusion  follows from the definition of the $U(c_i,\mu_i)$. As for the inclusion ``$\subseteq$'', it follows from the fact that if $z \in U(c,\mu)$, 
then 
$\lambda^{(i)}(z_i) = \mu_i$ holds for every $i$ and
 $\sum _{i=1}^{m} ( \inProd{\lambda^{(i)}(c_i)}{\lambda^{(i)}(z_i)}-\inProd{c_i}{z_i} ) = 0$ holds.
However, each term of the summation is nonnegative by item~$(ii)$ of Definition~\ref{def:FTvN}, so 
$\inProd{\lambda^{(i)}(c_i)}{\lambda^{(i)}(z_i)}=\inProd{c_i}{z_i}$ holds for all $i$ and we have $z_i \in U(c_i,\mu_i)$ for every $i$.

In the particular case where each $\jAlg_i$ is a Euclidean Jordan algebra, there is 
another way to see $\jAlg$ as a FTvN system employing the  eigenvalue map
\begin{equation}\label{eq:lambda_ord}
x\mapsto \lambda^\downarrow(x) \coloneqq (\lambda^{(1)}(x),\ldots,\lambda^{(m)}(x))^\downarrow. 
\end{equation}
In this case, the tuple $(\jAlg,\Re^r,\lambda^\downarrow)$ also forms a FTvN system, since $\jAlg$ is also a Jordan algebra and, in fact, $\lambda^\downarrow$ correspond to the same eigenvalue map discussed in Section~\ref{sec:jalg}.

Still, for $c=(c_1,\ldots,c_m) \in \jAlg$ and $\mu=(\mu_1,\ldots,\mu_m) \in \lambda^\downarrow(\jAlg)$, it makes sense to discuss how to compute an element in $U(c,\mu)$ using the  
$U(c_i,\mu_i)$, since in applications we would typically consider spectral decompositions according to the block structure of $\jAlg$.
In order to obtain an element in $U(c,\mu)$, we sort   the $r=r_1+\cdots+r_m$ real values in $\{\lambda^{(i)}_j(c_i)~|~i=1,\ldots,m;~j=1,\ldots,r_i\}$ (which are the eigenvalues of $c$) in descending order so that $\lambda^{(i_k)}_{j_k}(c_{i_k})$ denotes $k$-th largest eigenvalue of $c$.
Then, we have
\begin{equation}\label{eq:ucmu_ord}
\sum_{k=1}^{r}\mu_{i_k}\tilde{v}_{j_k}^{(i_k)} \in U(c,\mu),
\end{equation}
where $\tilde{v}_j^{(i)}$ is defined as
\[
\tilde{v}_j^{(i)}\coloneqq (0,\ldots,0,\underbrace{v_j^{(i)}}_{i\text{-th position}},0,\ldots,0),
\]
and  $\{{v}^{(i)}_1,\ldots, {v^{(i)}_{r_i}}\}$ is a Jordan frame for $c^{(i)}$.

The difference between 
$(\jAlg,\Re^r,\lambda^\downarrow)$ and $(\jAlg,\Re^r,\lambda)$ boils down to whether it is desirable to order all eigenvalues or to only order eigenvalues inside each block $\jAlg_i$. 
We will use both $(\jAlg,\Re^r,\lambda^\downarrow)$ and  $(\jAlg,\Re^r,\lambda)$ in our examples in Section~\ref{sec:ex}.

\section{Projected gradient and related topics}\label{sec:projg}

For a spectral set $\lambda^{-1}(\mathcal{C})$ with $\mathcal{C} \subseteq \Re^r$, it follows from Corollary~\ref{cor:proj-spec} that in order to project $x$ onto $\lambda^{-1}(\mathcal{C})$ it is enough to do as follows. 
First, we let $w$ be such that
\[
w \in \arg \min\, \{\norm{\lambda(x)-v}_2 \mid v \in \lambda(\jAlg), v \in \mathcal{C} \}.
\]
Then, we compute an element of $U(x,w)$ in order to obtain 
the desired projection. In particular, we have 
\[
\proj_{\lambda^{-1}(\mathcal{C})}(x) \supseteq U(x,w).
\]
Being able to compute the projection onto $\lambda^{-1}(\mathcal{C})$, allows us (in theory)
to use a simple  projected gradient algorithm for solving \eqref{eq:eig_prog}  by considering iterations that satisfy
\begin{equation}\label{eq:proj-grad-eigopt}
x_{k+1} \in \proj_{\lambda^{-1}(\mathcal{C})}(x_k-\alpha_k \nabla f(x_k)),
\end{equation}
where $\alpha_k>0$ is a step-size, see  Algorithm~\ref{alg:project}.

\begin{algorithm}[H] \label{alg:project}
\caption{Projected gradient method for \eqref{eq:eig_prog}}
	\KwIn{
		$f:\jAlg\to\Re$, $\mathcal{C}\subset\Re^r$, $x_0\in\jAlg$
	}
	$k \gets 0$
	
	\While{stopping criteria not satisfied}{
		Select a step-size $\alpha_k>0$.
		
		Set $y_k = x_k-\alpha_k\nabla f(x_k)$.
		
		Compute $w_k \in \arg \min_v \{\norm{\lambda(y_k)-v}_2 \mid v \in \lambda(\jAlg), v \in \mathcal{C} \}$.
		
		Compute $x_{k+1} \in U(y_k, w_k)$.
		
		
		$k \gets k+1$
	}
\end{algorithm}

First we analyze the convergence properties of Algorithm~\ref{alg:project} under the assumption that the problem data is semialgebraic, which is enough to cover a wide range of applications.

\begin{theorem}[\cite{ABS13}]\label{th:proj-grad-semialg}
Let $f:\Re^n\to \Re$ be a differentiable semialgebraic function with $L$-Lipschitz continuous gradient and $S\subset \Re^n$ be a closed semialgebraic set. Let $\{x_k\}$ be generated by
\[
x_{k+1} \in \proj_{S}(x_k-\alpha_k\nabla f(x_k))
\]
with step-size $\alpha_k \in (\ep,1/L-\ep)$ for some $\ep\in(0,1/(2L))$. 
If $\{x_k\}$ is bounded, then it converges to some $x^*$ such that $0 \in \nabla f(x^*)+N_{S}(x^*)$. Moreover, $\sum_k \norm{x_{k+1}-x_k}^2 < +\infty$.
\end{theorem}
\begin{proof}
This is a special case of Theorem~5.3 in \cite{ABS13} for semialgebraic functions. See also the observation following Theorem~5.3.
\end{proof}

\begin{theorem}\label{th:proj-grad}
Let $\mathcal{C}$ be a closed subset of $\Re^r$ and
$f:\jAlg \to \Re$ be a $C^1$ function with $L$-Lipschitz continuous gradient.
Suppose that $f$, $\lambda$ and $\mathcal{C}$ are semialgebraic. Let $\{x_k\}$ be generated by Algorithm~\ref{alg:project} with step-size $\alpha_k \in (\ep,1/L-\ep)$ for some $\ep\in(0,1/(2L))$. If $\{x_k\}$ is bounded then it converges to a critical point $x^*$ of $f + \delta _{\lambda^{-1}(\mathcal{C})}$, i.e., $0 \in \nabla f(x^*)+N_{\lambda^{-1}(\mathcal{C})}(x^*)$. Moreover, $\sum_k\norm{x_{k+1}-x_k}^2<+\infty$ holds.
\end{theorem}

\begin{proof}
Algorithm~\ref{alg:project} is the projected gradient method \eqref{eq:proj-grad-eigopt} for the constraint set $\lambda^{-1}(\mathcal{C})$.
Since $\mathcal{C}$ is semialgebraic, the spectral set
$\lambda^{-1}(\mathcal{C})$ is also semialgebraic, see \cite[Proposition~2.2.7]{Boc98}.
Hence, the assertion follows by Theorem~\ref{th:proj-grad-semialg}.
We also remark that since $f$ is $C^1$, 
we have $\partial(f + \delta _{\lambda^{-1}(\mathcal{C})} )(x^*) = \nabla f(x^*) +N_{\lambda^{-1}(\mathcal{C})}(x^*)$, by \cite[Exercise 8.8]{RW}.
\end{proof}

Next, we briefly comment on  local convergence rates for Algorithm~\ref{alg:project}, which, in general, requires  further assumptions. 
In what follows, let $F: \jAlg \to \Re\cup \{+\infty\}$ be a proper closed function.
Then, $F$ is said to be a \emph{Kurdyka-{\L}ojasiewicz (KL) function}  with exponent $\alpha$ (e.g., see \cite[Definition~2.2]{LP18}) at $\bar x \in \dom\partial F$ if there exists $c, \epsilon > 0$ and $\nu \in (0,\infty]$ so that 
\begin{equation}\label{eq:kl}
\dist(0, \partial F(x)) \geq c(F(x) - F(\bar x))^\alpha
\end{equation}
holds whenever $\norm{x-\bar{x}} \leq \epsilon$ and 
$F(\bar x) < F(x) < F(\bar{x}) + \nu$.
We also call \eqref{eq:kl} the \emph{KL property of $F$ at $\bar{x}$}.

Next, let $F$ be defined by $F(x) \coloneqq f(x) + \delta _{\lambda^{-1}(\mathcal{C})}(x)$, where we recall that $\delta _{\lambda^{-1}(\mathcal{C})}$ denotes the indicator function of $\lambda^{-1}(\mathcal{C})$.
If $F$ is a {KL function} with exponent $\alpha$, then a local convergence rate for Algorithm~\ref{alg:project} can be obtained through \cite[Theorem~3.4]{FGP15}. In particular, 
if $\alpha = 1/2$, the method has local linear convergence, see also the discussion in \cite[Section~5.3]{LP18}. We note this as a proposition.
\begin{proposition}\label{prop:local_linear}
Let $F \coloneqq f(\cdot) + \delta _{\lambda^{-1}(\mathcal{C})}(\cdot)$ 
and let $\{x_k\}$ be generated by Algorithm~\ref{alg:project} with step-size $\alpha_k \in (\ep,1/L-\ep)$ for some $\ep\in(0,1/(2L))$.
If $\{x_k\}$ converges to $\bar{x}$ and $F$ has the KL property with exponent $1/2$ at $\bar{x}$,  then $\{x_k\}$ converges linearly to $x^*$.
\end{proposition}
\begin{proof}
As remarked previously, this result follows from \cite[Theorem~3.4]{FGP15}, but here we discuss some of the details.
In reality, in order to invoke \cite[Theorem~3.4]{FGP15}, it is necessary to check that $x_k$ satisfy certain conditions which were labeled $H_1, H_2$ and $H_3$ therein. This is indeed the case, because $x_k$ is a particular case of the iteration discussed at the beginning of \cite[Section~4]{FGP15}, see Equation~(12) therein and the subsequent comments. 
The other assumption that needs to be checked is that $\{x_k\}$ \emph{$F$-converges} to $\bar{x}$, which means that $x_k \to \bar{x}$ and $F(x_k) \to F(\bar{x})$. 
The former holds by assumption and the latter holds because of the continuity of $f$  and the fact that $x_k \in \lambda^{-1}(\mathcal{C})$ for all $k$. 
This takes care of all the assumptions required to invoke case (ii) of \cite[Theorem~3.4]{FGP15} which leads to linear convergence, or in the parlance of \cite{FGP15}, \emph{exponential convergence}.
\end{proof}

Let us briefly examine the KL condition for $F$.
Since $f$ is smooth, we have $\partial F (x) = \nabla f(x) + \partial \delta _{\lambda^{-1}(\mathcal{C})}(x)$, see \cite[Exercise 8.8]{RW}. We recall, 
however, that $\partial \delta _{\lambda^{-1}(C)}(x)$ is the 
normal cone of $\lambda^{-1}(\mathcal{C})$ at $x$. So, for $x \in \lambda^{-1}(\mathcal{C})$, \eqref{eq:kl} can be written as 
\begin{equation}\label{eq:kl2}
\dist(\nabla f(x), -N_{\lambda^{-1}(\mathcal{C})}(x)) \geq c(f(x) - f(\bar x))^\alpha
\end{equation}
In view of \eqref{eq:kl2}, it is important to analyze 
the normal cone $N_{\lambda^{-1}(\mathcal{C})}(x)$. In the particular 
case where $\jAlg$ is a Jordan algebra and $\mathcal{C} \subseteq \Re^r$ is a permutation invariant set\footnote{That is, $P(\mathcal{C}) = \mathcal{C}$ holds for any $r\times r$ permutation matrix $P$.}, the computation of 
the normal cone $N_{\lambda^{-1}(\mathcal{C})}(x)$ can be related to the normal cone of $\mathcal{C}$ at $\lambda(x)$ as follows.

\begin{proposition}
Let $\jAlg$ be a Euclidean Jordan algebra of rank $r$ and let $\mathcal{C} \subseteq \Re^r$ be a permutation invariant closed set. Then 
\[
N_{\lambda^{-1}(\mathcal{C})}(x) = \{s \in \jAlg \mid \exists \mathcal{J} \in \mathcal{J}({x,s}) \text{ with } \diag (s,\mathcal{J}) \in N_\mathcal{C}(\lambda(x)) \},
\]
where $\mathcal{J}(x,s)$ is the set of Jordan frames $\{e_1, \ldots, e_r\}$ for which $x = \lambda_1(x)e_1 + \cdots \lambda_r(x) e_r$ and $s = a_1 e_1 + \cdots a_r e_r$ hold\footnote{Because the eigenvalues are unique, the $a_i$'s must be the eigenvalues of $s$. Overall, $\{e_1, \ldots, e_r\}$ is a Jordan frame that ``diagonalizes'' both $x$ and $s$, but only the eigenvalues of $x$ are required to be in nonincreasing order.} and $\diag(s,\mathcal{J}) $ is the vector $(a_1,\ldots, a_r)$.
\end{proposition}
\begin{proof}
It follows directly from \cite[Theorem~27]{LT20} applied to $\delta _{\mathcal{C}}$ since $\delta _{\mathcal{C}} \circ \lambda$ is the indicator function of $\lambda^{-1}(\mathcal{\mathcal{C}})$ and $N_{\mathcal{C}}(u) = \partial \delta _\mathcal{\mathcal{C}}$.
\end{proof}


\subsection{Application to feasibility problems and convergence rates}\label{sec:feasp}
In this section, we consider a feasibility problem of the following form:
\begin{equation}\label{eq:feas-probl}
\text{Find } x \in \bar{\mathcal{C}} \cap \lambda^{-1}(\mathcal{C}),
\end{equation}
where $(\jAlg, \Re^r, \lambda)$ is a FTvN system, $\mathcal{C} \subset \Re^r$ and $\bar{\mathcal{C}} \subset \jAlg$ is a convex set for which $\proj_{\bar{\mathcal{C}}}(\cdot)$ is assumed to be available.
This can be reformulated as the following optimization problem.
\begin{align}\label{eq:feas-min}
\begin{split}
{\min _{x \in \jAlg}} & \quad f(x):= \frac{1}{2}\dist(x,\bar{\mathcal{C}})^2\\ 
\mbox{subject to} & \quad \lambda(x) \in \mathcal{C}.
\end{split}
\end{align}
The gradient of $f$  is given by
\[
\nabla f(x) = x-\proj_{\bar{\mathcal{C}}}(x),
\]
and if $\bar{\mathcal{C}}$ is semialgebraic then 
$f$ is semialgebraic as well \cite[Proposition~2.2.8]{Boc98}.
Note that $\nabla f$ is $1$-Lipschitz continuous. Therefore, we can use Algorithm~\ref{alg:project},
where the resulting iterative scheme can be described as 
\begin{equation}\label{eq:alt-proj}
y_k \coloneqq (1-\alpha_k)x_k + \alpha_k\proj_{\bar{\mathcal{C}}}(x_k),\quad x_{k+1} \in \proj_{\lambda^{-1}(\mathcal{C})}(y_k),\quad k=0,1,2,\ldots.
\end{equation}
In particular, when $\alpha_k \in (\ep,1-\ep)$ holds for some $\ep\in (0,1/2)$ and $f$ is semialgebraic, we obtain the convergence guarantee described in Theorem~\ref{th:proj-grad}.

Regarding convergence rates, it seems natural that this must depend on the geometry of the intersection between 
$\mathcal{\bar{C}}$ and $\lambda^{-1}(\mathcal{C})$. 
In what follows, we will verify that if the intersection is \emph{transversal at $x$}, then $f+ \delta _{\lambda^{-1}(\mathcal{C})}$ satisfy the KL property with exponent $1/2$ at $x$. We recall that  $S_1, S_2 \subseteq \jAlg$ are said to be to be \emph{transversal} at $\bar{x}$ if
\begin{equation}\label{eq:trans}
N_{S_1}(\bar{x}) \cap (-N_{S_2}(\bar{x})) = \{0\}.
\end{equation}
Transversality is related to well-known constraint qualifications used in several contexts. With that we have the following result. Note that this will not require that $\lambda$ or $\mathcal{\bar C}$ be semialgebraic.

\begin{proposition}\label{prop:trans}
If the iteration \eqref{eq:alt-proj} converges to a critical point $\bar{x}$  of 
$f = \frac{1}{2}\dist(\cdot,\bar{\mathcal{C}})^2 + \delta _{\lambda^{-1}(\mathcal{C})}(\cdot)$ 
and $f$ satisfies the KL property with exponent $1/2$ at $\bar{x}$, then 
the convergence rate is linear. 
In particular, this is true 
if $\bar{x} \in  \mathcal{\bar{C}} \cap \lambda^{-1}(\mathcal{C})$ and
$\mathcal{\bar{C}}$ and $\lambda^{-1}(\mathcal{C})$ are  transversal at $\bar{x}$.
\end{proposition}
\begin{proof}
If $\mathcal{\bar{C}}$ and $\lambda^{-1}(\mathcal{C})$ intersect transversally at $\bar{x}$, 
then $f$ is a KL function with exponent $1/2$  by \cite[Theorem~5]{CPTZ20}\footnote{The correspondence between the notation in \cite[Theorem~5]{CPTZ20} and our case is as follows. $A$ is the identity map, $D$ is $\mathcal{\bar{C}}$ and $\mathcal{C}$ is $\delta_{\lambda^{-1}(\mathcal{C})}$. }. In any case, the result follows 
from Proposition~\ref{prop:local_linear}.
\end{proof}

It is interesting to note that the limiting case  where $\alpha_k$ is taken to be $1$ in \eqref{eq:alt-proj} corresponds to the \emph{alternating projections algorithm}, where each iteration consists of successively projecting onto $\bar{\mathcal{C}}$ and $\lambda^{-1}(C)$. 
Although the alternating projections algorithm falls outside of the scope of Theorem~\ref{th:proj-grad}, it is known 
that it locally converges with a linear rate under a weaker assumption called \emph{intrinsical transversality}, see \cite[Theorem~6.1]{DIL15} and \cite{NR15} for a related work.
On the other hand, it is not clear if the global guarantee in Theorem~\ref{th:proj-grad} also holds for alternating projections under semialgebraic assumptions.

\paragraph{The convex case}
In the special case where $ \mathcal{\bar C}$ and 
$\lambda^{-1}(\mathcal{C})$ are convex the situation is far more favourable. First of all, the iteration \eqref{eq:alt-proj} always converges to a point in the intersection $ \mathcal{\bar C} \cap \lambda^{-1}(\mathcal{C})$, provided that the problem is feasible. 
This is a consequence of classical convergence results of the proximal gradient method for convex functions, but it can also be obtained by observing that the iteration in \eqref{eq:alt-proj} falls under the scope of several frameworks for analyzing convex feasibility problems and beyond. These frameworks are also able to take care of the limiting case $\alpha _k = 1$, e.g., \cite[pg.~2]{BB96}, \cite{blt17}. 

In the convex case, the interesting question is not \emph{whether} \eqref{eq:alt-proj} converges but \emph{how fast does it converge}. Under convexity, it is known that the convergence rate of many algorithms is related to the underlying error bound that holds between sets. In particular, when a so-called \emph{Lipschitz error bound} (also called \emph{bounded linear regularity} in some papers, e.g., \cite{BB96}) holds,  several algorithms are known to converge at a linear rate.
What is somewhat less known is that under even weaker error bound assumptions \emph{concrete}\footnote{By \emph{concrete} we mean that the convergence rate is upper bounded in terms of a function of the iteration number $k$ (e.g., see \cite[Theorem~3.1]{blt17}). 
} 	convergence rates are still obtainable. For example, see \cite[Section~3]{blt17} for the case of H\"olderian error bounds in the setting of fixed point problems or \cite[Sections~4 and 5]{LL22} for the case of general error bounds for convex feasibility problems.

Given that the convex case is easier to handle, a related question is to understand 
when does $\lambda^{-1}(\mathcal{C})$ becomes a convex set. In the case of Jordan algebras this is very well-understood. In particular, 
if $\mathcal{C}$ is permutation invariant, 
then $\lambda^{-1}(\mathcal{C})$ is convex if and only if $\mathcal{C}$ is convex \cite[Theorem~3]{JG17}, see also \cite[Theorem~27]{baes07}.
In the general case of FTvN systems, the situation is more subtle and additional conditions seem necessary in order to obtain similar results, see \cite{JG23}.

\section{Eigenvalue programming examples}\label{sec:ex}

In this section, we discuss two examples of the eigenvalue programming problem \eqref{eq:eig_prog} and show numerical experiments for the projected gradient method described in Algorithm~\ref{alg:project}. 
In all examples, the problem data are semialgebraic so we are under the scope of Theorem~\ref{th:proj-grad}.
All experiments are implemented in Julia and conducted on a 3.0GHz Intel Xeon E5-1680v2 processor with 64GB of RAM.\footnote{The implementation is available in \url{https://github.com/ito-masaru/pgm-eig-prog}}

\subsection{Inverse eigenvalue problems}\label{sec:inv}
Given a FTvN system $(\jAlg, \Re^r, \lambda)$, $\lambda^* \in \lambda(\jAlg)$ and $a_i \in \jAlg$ ($i=0,1,\ldots,d$), we consider the \emph{inverse eigenvalue problem}:
\begin{equation}\label{eq:gen-IEQ}
\text{Find } c \in \Re^d \text{ such that }
\lambda(a_0 + c_1a_1 + \cdots + c_da_d) = \lambda^*.
\end{equation}
For the special case $\jAlg=\S^n$, this problem has been extensively researched \cite{Friedland87,chen_chu_96, chu98,chu_golub_2002,chu_golub_2005}. However, beyond $\S^n$, we are not aware of a systematic study of inverse eigenvalue problems. 

The problem \eqref{eq:gen-IEQ} is equivalent to the feasibility problem \eqref{eq:feas-probl} with
$$
\bar{\mathcal{C}} \coloneqq a_0 + \spanVec\{a_1,\ldots,a_d\},\quad \mathcal{C} \coloneqq \{\lambda^*\}.
$$
In this case, the iterative scheme \eqref{eq:alt-proj} reduces to
\begin{equation}\label{eq:alg-IEP}
y_k = (1-\alpha_k)x_k + \alpha_k\proj_{\bar{\mathcal{C}}}(x_k),\quad x_{k+1} \in U(y_k,\lambda^*),
\end{equation}
since $\proj_{\lambda^{-1}(\mathcal{C})}(y) = U(y,\lambda^*)$ holds by Corollary~\ref{cor:proj-spec}.

\subsubsection{Numerical experiment}
Here we show some numerical results of the method \eqref{eq:alg-IEP} applied to the inverse eigenvalue problem \eqref{eq:gen-IEQ} in the case
\begin{equation}\label{eq:alg-num-IEP}
\jAlg=\underbrace{\Re^{n+1} \times \cdots \times \Re^{n+1}}_{m \text{ times}} \times \underbrace{\mathcal{S}^n\times \cdots \times \mathcal{S}^n}_{\ell \text{ times}},
\end{equation}
where we regard $\Re^{n+1}$ and $\S^n$ as the Jordan algebras in Examples~\ref{ex:soc} and~\ref{ex:sym}, respectively. 
We examine both the block-wise eigenvalue maps $\lambda$ and the ordered one $\lambda^\downarrow$ described in \eqref{eq:prod-FTvN} and \eqref{eq:lambda_ord}, respectively.
The goal of this experiment is to check the behavior of Algorithm~\ref{alg:project} and see how far an initial point must be in order to ensure convergence to a solution of \eqref{eq:gen-IEQ}.

Given $(\ell,m,n,d)$, we examine 10 random instances of the problem \eqref{eq:alg-num-IEP} under the FTvN system \eqref{eq:alg-num-IEP} by setting $\lambda^* \coloneqq \lambda(a_0+c_1a_1+\cdots+c_da_d) \in \jAlg$, where $c_i \in \Re$ and components of $a_i \in \jAlg$ are generated from the uniform distribution over $[0,1]$. We run the method \eqref{eq:alg-IEP} with step-size $\alpha_k=0.99$ until it finds an iterate $x_k$ satisfying the termination criterion
\[
\dist(x_k, \bar{\mathcal{C}}) \leq 10^{-3}
\]
(note that we have $\lambda(x_k)=\lambda^*$ for $k\geq 1$ by construction).
If this condition fails to be satisfied after $10000$ iterations, we restart the method and change the initial point. 
The initial point $x_0$ is chosen as follows:
\[
x_0:=x^* + 100 \norm{x^*} u/2^r,
\]
where $x^*\coloneqq a_0+c_1a_1+\cdots+c_ma_m$ is an optimal solution, $u \in \jAlg$ is generated from the uniform distribution on the unit sphere of $\jAlg$ and $r \geq 0$ is the number of restarts of the method. Note that $r$ controls the relative distance from $x_0$ to the optimal solution since $\norm{x_0-x^*}/\norm{x^*}=100/2^r$ holds.

Tables~\ref{table:IEP} and \ref{table:IEP-order} show numerical results for some problem instances under the eigenvalue maps $\lambda$ and $\lambda^\downarrow$, respectively. For each $(\ell,m,n)$, the dimension $d$ is chosen as $d=\lfloor\dim(\jAlg)/\rho\rfloor$ for $\rho \in \{0.2, 0.4, 0.6, 0.8\}$. We see that the proposed method successfully converges to an optimal solution (i.e., a feasible solution to \eqref{eq:gen-IEQ}) in many cases, while some instances require several restarts. Also employing the block-wise eigenvalue map $\lambda$ is relatively better than the ordered one $\lambda^\downarrow$ in average.
Regarding the performance, we can see that the number of iterations is decreasing with respect to the parameter $d$. 
This is somewhat puzzling and at this point we can only speculate on the reasons for that. 
A possibility is that as  $d$ increases, the dimension of $\bar{\mathcal{C}}$ increases as well.
Given the random nature of the instances, $\bar{\mathcal{C}} \cap \lambda^{-1}(\{\lambda^*\})$ might get be ``better-conditioned'' as $d$ increases. For example, it could be case that the region corresponding to linear convergence as in Propositions~\ref{prop:local_linear} and \ref{prop:trans} gets larger.

\begin{table}[htbp]
\centering
\begin{tabular}{cccc|cccc|cccc}
\hline

\multirow{2}{*}{$\ell$} & \multirow{2}{*}{$m$} & \multirow{2}{*}{$n$} & \multirow{2}{*}{$d$} & \multicolumn{4}{|c|}{Iterations} & \multicolumn{4}{|c}{Restarts}  \\
    &    &     &     & mean & max & min & std & mean & max & min & std \\
\hline\hline

1 & 0 & 10 & 11 & 1836.4 & 4136 & 574 & 1205.8 & 0.8 & 5 & 0 & 1.8\\
1 & 0 & 10 & 22 & 107.0 & 312 & 34 & 78.6 & 0.0 & 0 & 0 & 0.0\\
1 & 0 & 10 & 33 & 31.6 & 58 & 20 & 11.5 & 0.0 & 0 & 0 & 0.0\\
1 & 0 & 10 & 44 & 19.2 & 45 & 10 & 10.7 & 0.0 & 0 & 0 & 0.0\\\hline
1 & 1 & 10 & 13 & 2255.9 & 4110 & 336 & 1135.0 & 3.1 & 7 & 0 & 3.4\\
1 & 1 & 10 & 26 & 82.4 & 152 & 47 & 37.6 & 0.0 & 0 & 0 & 0.0\\
1 & 1 & 10 & 39 & 47.1 & 63 & 29 & 10.8 & 0.0 & 0 & 0 & 0.0\\
1 & 1 & 10 & 52 & 32.0 & 56 & 15 & 12.2 & 0.0 & 0 & 0 & 0.0\\\hline
1 & 5 & 10 & 22 & 3049.5 & 9666 & 564 & 2983.9 & 2.3 & 7 & 0 & 3.1\\
1 & 5 & 10 & 44 & 102.6 & 189 & 52 & 37.6 & 0.0 & 0 & 0 & 0.0\\
1 & 5 & 10 & 66 & 59.2 & 95 & 28 & 18.6 & 0.0 & 0 & 0 & 0.0\\
1 & 5 & 10 & 88 & 26.4 & 42 & 16 & 7.2 & 0.0 & 0 & 0 & 0.0\\\hline
3 & 0 & 10 & 33 &5328.8 & 9480 & 1507 & 2515.0 & 4.7 & 9 & 0 & 3.7\\
3 & 0 & 10 & 66 &110.3 & 162 & 67 & 30.6 & 0.0 & 0 & 0 & 0.0\\
3 & 0 & 10 & 99 &50.7 & 65 & 41 & 8.5 & 0.0 & 0 & 0 & 0.0\\
3 & 0 & 10 & 132 &56.8 & 93 & 25 & 20.8 & 0.0 & 0 & 0 & 0.0\\\hline
3 & 1 & 10 & 35 &4497.7 & 7025 & 1241 & 2059.6 & 5.6 & 10 & 0 & 3.7\\
3 & 1 & 10 & 70 &132.9 & 218 & 96 & 34.6 & 0.0 & 0 & 0 & 0.0\\
3 & 1 & 10 & 105 &57.2 & 86 & 37 & 16.2 & 0.0 & 0 & 0 & 0.0\\
3 & 1 & 10 & 140 &42.5 & 69 & 33 & 11.2 & 0.0 & 0 & 0 & 0.0\\\hline
3 & 5 & 10 & 44 &5138.8 & 9206 & 1325 & 2886.4 & 3.9 & 9 & 0 & 3.6\\
3 & 5 & 10 & 88 &127.8 & 196 & 84 & 33.5 & 0.0 & 0 & 0 & 0.0\\
3 & 5 & 10 & 132 &79.5 & 106 & 54 & 16.4 & 0.0 & 0 & 0 & 0.0\\
3 & 5 & 10 & 176 &56.1 & 77 & 40 & 11.0 & 0.0 & 0 & 0 & 0.0\\
\hline
\end{tabular}
\caption{Numerical results for problem \eqref{eq:gen-IEQ} under the block-wise eigenvalue map $\lambda$. The column ``Iterations'' collects the summary of the number of iterations at the final run of the proposed method, showing the mean, the maximum, the minimum and the sample standard deviation. The summary of the number of restarts is given in the column ``Restarts''.}
\label{table:IEP}
\end{table}

\begin{table}[htbp]
\centering
\begin{tabular}{cccc|cccc|cccc}
\hline

\multirow{2}{*}{$\ell$} & \multirow{2}{*}{$m$} & \multirow{2}{*}{$n$} & \multirow{2}{*}{$d$} & \multicolumn{4}{|c|}{Iterations} & \multicolumn{4}{|c}{Restarts}  \\
    &    &     &     & mean & max & min & std & mean & max & min & std \\
\hline\hline
1 & 0 & 10 & 11 & 1836.4 & 4136 & 574 & 1205.8 & 0.8 & 5 & 0 & 1.8\\
1 & 0 & 10 & 22 & 107.0 & 312 & 34 & 78.6 & 0.0 & 0 & 0 & 0.0\\
1 & 0 & 10 & 33 & 31.6 & 58 & 20 & 11.5 & 0.0 & 0 & 0 & 0.0\\
1 & 0 & 10 & 44 & 19.2 & 45 & 10 & 10.7 & 0.0 & 0 & 0 & 0.0\\\hline
1 & 1 & 10 & 13 & 2255.9 & 4110 & 336 & 1135.0 & 3.1 & 7 & 0 & 3.4\\
1 & 1 & 10 & 26 & 82.4 & 152 & 47 & 37.6 & 0.0 & 0 & 0 & 0.0\\
1 & 1 & 10 & 39 & 56.8 & 160 & 29 & 37.4 & 0.0 & 0 & 0 & 0.0\\
1 & 1 & 10 & 52 & 39.8 & 64 & 24 & 12.6 & 0.0 & 0 & 0 & 0.0\\\hline
1 & 5 & 10 & 22 & 2049.6 & 4340 & 494 & 1477.4 & 3.0 & 10 & 0 & 3.9\\
1 & 5 & 10 & 44 & 1937.7 & 8993 & 74 & 2856.9 & 1.5 & 6 & 0 & 2.5\\
1 & 5 & 10 & 66 & 78.8 & 127 & 46 & 30.5 & 0.0 & 0 & 0 & 0.0\\
1 & 5 & 10 & 88 & 24.6 & 36 & 15 & 6.3 & 0.0 & 0 & 0 & 0.0\\\hline
3 & 0 & 10 & 33 &5426.8 & 8548 & 1950 & 2319.9 & 6.3 & 11 & 0 & 4.3\\
3 & 0 & 10 & 66 &91.5 & 143 & 62 & 23.2 & 0.0 & 0 & 0 & 0.0\\
3 & 0 & 10 & 99 &49.4 & 62 & 36 & 10.2 & 0.0 & 0 & 0 & 0.0\\
3 & 0 & 10 & 132 &56.5 & 100 & 21 & 25.5 & 0.0 & 0 & 0 & 0.0\\\hline
3 & 1 & 10 & 35 &4907.9 & 9901 & 1943 & 2944.5 & 7.3 & 12 & 1 & 4.2\\
3 & 1 & 10 & 70 &110.1 & 141 & 74 & 23.4 & 0.8 & 5 & 0 & 1.7\\
3 & 1 & 10 & 105 &727.5 & 3546 & 30 & 1201.4 & 1.8 & 6 & 0 & 2.4\\
3 & 1 & 10 & 140 &119.3 & 278 & 47 & 81.1 & 0.0 & 0 & 0 & 0.0\\\hline
3 & 5 & 10 & 44 &4351.8 & 7261 & 1221 & 2172.1 & 9.3 & 12 & 7 & 1.6\\
3 & 5 & 10 & 88 &142.3 & 208 & 77 & 41.2 & 2.3 & 5 & 0 & 2.1\\
3 & 5 & 10 & 132 &950.1 & 2505 & 213 & 682.6 & 0.0 & 0 & 0 & 0.0\\
3 & 5 & 10 & 176 &62.1 & 103 & 44 & 16.8 & 0.0 & 0 & 0 & 0.0\\
\hline
\end{tabular}
\caption{Numerical results for problem \eqref{eq:gen-IEQ} under the ordered eigenvalue map $\lambda^\downarrow$}
\label{table:IEP-order}
\end{table}

\subsection{Vanishing Quadratic Constraints}\label{sec:vqc}
In this section we take a look at the following problem.
Given $m$ convex quadratic inequalities and an integer $\ell \in \{0,1,\ldots,m\}$, we would like to find a common solution $x\in \Re^n$ that satisfies at least $\ell$ with equality as follows.
\begin{subnumcases}{\label{eq:vqc}}
	\norm{A_i x +b_i}_2 \leq \innprod{c_i}{x} + d_i, & $i=1,\ldots,m,$ \label{eq:vqc1} \\
	\text{At least } \ell \text{ of these inequalities are tight.}  \label{eq:vqc2}
\end{subnumcases}
Here, $A_i\in \Re^{n_i\times n}$, $b_i \in \Re^{n_i}$, $c_i \in \Re^n$, $d_i \in \Re$ and $\ell \in \{0,1,\ldots,m\}$ are given. We refer to a constraint in format of the \eqref{eq:vqc1} and \eqref{eq:vqc2} as \emph{quadratic inequality with a vanishing constraint} or a \emph{vanishing quadratic constraint} for short.
In this subsection, we will discuss how to reformulate a vanishing quadratic constraint in the form of a feasibility problem as in \eqref{eq:feas-probl}.

Let $\jAlg \coloneqq \Re^{n_1+1}\times \cdots \times \Re^{n_m+1}$ and $\stdCone \coloneqq \SOC{2}{n_1+1}\times \cdots \times \SOC{2}{n_m+1}$ be a direct product of second-order cones. We consider that each $\Re^{n_i+1}$ is a Euclidean Jordan algebra as in Example~\ref{ex:soc} and the corresponding cone is $\SOC{2}{n_i+1}$.
With that, $(\jAlg, \Re^{2m},\lambda^\downarrow)$ is a FTvN system where $\lambda^\downarrow$ is as in \eqref{eq:lambda_ord}.  

Next, we define
\[b\coloneqq(b_1, d_1, \ldots, b_m,d_m) \in \jAlg,\qquad A:\Re^n\to\jAlg,~~ Ax \coloneqq (A_1x, \innprod{c_1}{x},\ldots, A_mx, \innprod{c_m}{x}).
\]
We have the following result.
\begin{theorem}[Vanishing quadratic constraints as a rank constraint]\label{theo:vqc}
Suppose that $x \in \Re^n$ satisfies \eqref{eq:vqc1} and \eqref{eq:vqc2}, then 
\begin{equation}\label{eq:vqc_t}
Ax+b \in  \stdCone\quad \text{ and }\quad \matRank(Ax+b) \leq 2m-\ell.
\end{equation}
Conversely, if $x \in \Re^n$ satisfies \eqref{eq:vqc_t}
and $(A_ix+b_i,\innprod{c_i}{x}+d_i) \ne 0$ holds for all $i$, then  $x$ satisfies \eqref{eq:vqc1} and \eqref{eq:vqc2}.
\end{theorem}
\begin{proof}
First, suppose that  $x \in \Re^n$ satisfies \eqref{eq:vqc1} and \eqref{eq:vqc2}. From, \eqref{eq:vqc1} we obtain 
$Ax+b \in \stdCone$.
Next, let $z \coloneqq Ax+b$, so that $z = (z_1,t_1\ldots, z_m,t_m)$ can be divided in $m$ blocks 
so that \[
(z_i,t_i)  = (A_ix+b_i, \innprod{c_i}{x}+d_i)\] 
holds for $i \in \{1, \ldots, m\}$. 
We have that $\norm{A_i x +b_i}_2 = \innprod{c_i}{x}+d_i$
if and only if $(z_i,t_i)$ belongs to the boundary of $\SOC{2}{n_i+1}$. 

Equivalently, the block $(z_i,t_i) \in \SOC{2}{n_i+1}$ belongs to the boundary of $\SOC{2}{n_i+1}$ if and only if the smallest eigenvalue (i.e., $\lambda_{-}$ in \eqref{eq:spect_dec}) of  $(z_i,t_i) $ vanishes. 
Therefore, the rank of $(z_i,t_i)$ in the algebra $\Re^{n_i+1}$ must be $0$ or $1$. 
Since the rank of $z$ is the sum of ranks of the individual blocks, \eqref{eq:vqc2} implies that $\rank(z) \leq \ell + 2(m-\ell) = 2m -\ell$ holds.

Conversely, suppose that $x$ satisfies \eqref{eq:vqc_t}.
Then, in particular, $Ax + b \in \stdCone$ holds, which, by definition implies that that \eqref{eq:vqc1} is satisfied. 
Under the assumption that  $Ax + b \in \stdCone$ and $(A_ix+b_i,\innprod{c_i}{x}+d_i)$
is never zero, we have that the rank of $(A_ix+b_i,\innprod{c_i}{x}+d_i)$ is either $1$ or $2$ and the former happens if and only if  $\norm{A_i x +b_i}_2 = \innprod{c_i}{x} + d_i$.
Therefore, if $\rank(Ax + b) \leq 2m - \ell$, then at least $\ell$ of the inequalities are tight, since otherwise we would have $\rank(Ax + b) \geq 2(m-(\ell+1)) + (\ell+1)> 2m-\ell$.
\end{proof}

%
If $\ell \geq 1$, we have equivalently that 
$x$ satisfies \eqref{eq:vqc_t} if and only if
\[
\lambda^\downarrow _{2m}(Ax+b) = \cdots = \lambda^\downarrow_{2m-\ell+1}(Ax+b) = 0.
\]
The final step to reformulate the problem in the format described in Section~\ref{sec:feasp} is to perform the variable transformation 
$y = Ax + b$. This leads ot the the following feasibility problem
\begin{equation}\label{eq:vqc3}
\mathrm{find}\quad y \in (\matRange(A)+b) \cap ( \stdCone \cap \{y \mid \rank(y) \leq 2m -\ell\} ).
\end{equation}
Finally, if we let $\bar{\mathcal{C}} \coloneqq \matRange(A)+b$, $\mathcal{C} \coloneqq \{u \in \Re^{2m}_{+} \mid u _{2m} = \cdots = u_{2m-\ell+1} = 0 \}$, we get the following equivalent problem
\begin{equation}\label{eq:vqc4}
\mathrm{find}\quad y \in \bar{\mathcal{C}} \quad \mathrm{and}\quad \lambda^\downarrow(y) \in  \mathcal{C},
\end{equation}
which can be solved as described in Section~\ref{sec:feasp}. 

We remark that the projection of $v \in \Re^r_{\downarrow}$ onto $\mathcal{C}$ is easily computable as
\[\proj_{\mathcal{C}}(v) = (\max(v_1,0),\ldots,\max(v_{2m-\ell},0),0,\ldots,0),\quad \forall v \in \Re^r_{\downarrow}.
\]
Therefore, using Corollary~\ref{cor:proj-spec}, the iterative scheme \eqref{eq:alt-proj} reduces to
\begin{equation*}
z_k = (1-\alpha_k)y_k + \alpha_k\proj_{\bar{\mathcal{C}}}(y_k),\quad y_{k+1} \in U(z_k,\proj_{\mathcal{C}}(\lambda^\downarrow(z_k)) ).
\end{equation*}
By its turn, an element of $U(z_k,\proj_{\mathcal{C}}(\lambda^\downarrow(z_k)) )$ can be obtained explicitly through \eqref{eq:ucmu_ord} and the spectral decomposition in the algebra of second-order cones as described in Example~\ref{ex:soc}.


A final observation is that in order to obtain a ``if and only if'' correspondence between \eqref{eq:vqc3} and the pair \eqref{eq:vqc1} and \eqref{eq:vqc2},  Theorem~\ref{theo:vqc} requires the 
assumption that $(A_ix+b_i,\innprod{c_i}{x}+d_i) \ne 0$ holds for all $i$ whenever $x$ satisfies \eqref{eq:vqc1} and \eqref{eq:vqc2}. 
An important case where this assumption is satisfied is when all the $c_i$'s are zero and the $d_i$'s are nonzero as in the next subsection.


\subsubsection{Intersection of ellipsoids and a numerical experiment}\label{sec:ellipsoids}
As a special case of \eqref{eq:vqc}, 
suppose we are given $m$ ellipsoids and we would like to find an intersection point that lies in the boundary of at least $\ell$ of them.
Formally, given positive definite matrices $Q_i \in \S^n$ and $p_i \in \Re^n$ for $i = 1, \ldots, m$, our problem is to 
find $x \in \Re^n$ satisfying 
\begin{subnumcases}{\label{eq:ei}} 
(x-p_i)^TQ_i(x-p_i) \leq 1 & $i=1,\ldots,m,$ \label{eq:ei1} \\
\text{At least } \ell \text{ of these inequalities are tight.}  \label{eq:ei2}
\end{subnumcases}
The corresponding $A$ and $b$ are given by
\begin{equation}\label{eq:Ab-ellip}
b\coloneqq (-Q_1^{1/2}p_1,1,\ldots,-Q_m^{1/2}p_m,1), \qquad Ax \coloneqq (Q_1^{1/2}x, 0,\cdots,Q_m^{1/2}x,0).
\end{equation}
We can rewrite \eqref{eq:ei} as
\begin{numcases}{}
Ax+b \in  \stdCone\coloneqq \underbrace{\SOC{2}{n+1}\times \cdots \times \SOC{2}{n+1}}_{m \text{ times}} \notag\\
\matRank(Ax+b) \leq 2m-\ell \notag 
\end{numcases}
which is equivalent to 
\begin{equation}\label{eq:vc}
\mathrm{find}\quad y \in \bar{\mathcal{C}} \quad \mathrm{and}\quad \lambda^\downarrow(y) \in  \mathcal{C},
\end{equation}
where $\bar{\mathcal{C}} \coloneqq \matRange(A)+b$ and $\mathcal{C} \coloneqq \{u \in \Re^{2m}_{\downarrow} \mid u _{2m} = \cdots = u_{2m-\ell+1} = 0 \}$.
We observe that in this case, the ``$c_i$'s'' and ``$d_i$'s''in Theorem~\ref{theo:vqc} are all $0$ and $1$, respectively. 
So, indeed, \eqref{eq:vc} and the pair \eqref{eq:ei1}, \eqref{eq:ei2} are equivalent.

\paragraph{Recovering $x$ from $y$}
Here we remark that a solution to \eqref{eq:ei} can be easily recovered from a solution to \eqref{eq:vc}.
Let $y=(\bar{y}_1,y_{10},\ldots,\bar{y}_m,y_{m0}) \in \bar{\mathcal{C}}$. Let us check that there is a unique $x$ satisfying $y=Ax+b$.
By \eqref{eq:Ab-ellip}, the relation $y=Ax+b$ is equivalent to
\[
\bar{y}_{i} = Q_i^{1/2}(x-p_i),\quad y_{i0}=1,\quad i=1,\ldots,m.
\]
The former equation is equivalent to
\[
x = Q_1^{-1/2}\bar{y}_1 + p_1 = \cdots = Q_m^{-1/2}\bar{y}_m + p_m.
\]
Numerically, we can also average the different expressions to $x$ which leads to the following formula.
\begin{equation}\label{eq:recov-x}
x = \frac{1}{m}\sum_{i=1}^m (Q_i^{-1/2}\bar{y}_i + p_i).
\end{equation}

\paragraph{Numerical results}
Here, we show numerical results of
applying Algorithm~\ref{alg:project} to \eqref{eq:feas-min} in order the solve the problem \eqref{eq:vc} in the case $(n,m,\ell)=(2,3,1),(2,3,2)$. We generated randomly three ellipsoids centered at the origin as illustrated in Figure~\ref{fig:ellip}.
We use constant step-size $\alpha_k=0.99$. Let $\{y^k\}$ be the generated sequence and $x^k$ be the approximate solution recovered by $y^k$ using the formula \eqref{eq:recov-x}. We terminate the algorithm if
$$
\dist(y^k, \bar{\mathcal{C}}) \leq 10^{-3}
$$
(Note that $\dist(\lambda^\downarrow(y^k), \mathcal{C})=0$ holds for $k\geq 1$ by construction).
In Figure~\ref{fig:ellip} we demonstrate the numerical results.
\begin{itemize}
\item The case $\ell=1$ in \eqref{eq:ei} corresponds to the problem of finding a boundary point on the intersection of given $m$ ellipsoids. Figure \ref{fig:ellip} shows the behavior of $\{x^k\}$ from various initial points. One can see that the trajectory $\{x^k\}$ successfully converges to a point on the boundary.
\item The case $\ell=2$ in \eqref{eq:ei} corresponds to finding a ``degenerate'' boundary point  of the intersection of given $m$ ellipsoids. Figure \ref{fig:ellip} illustrates that $\{x^k\}$ successfully converges to a desired point except the one starting from the origin. We remark that the origin yields a stationary point for the corresponding problem \eqref{eq:feas-min} in this case.
\end{itemize}

\begin{figure}[htbp]
  \begin{minipage}{0.5\hsize}
  \begin{center}
   \includegraphics[width=75mm]{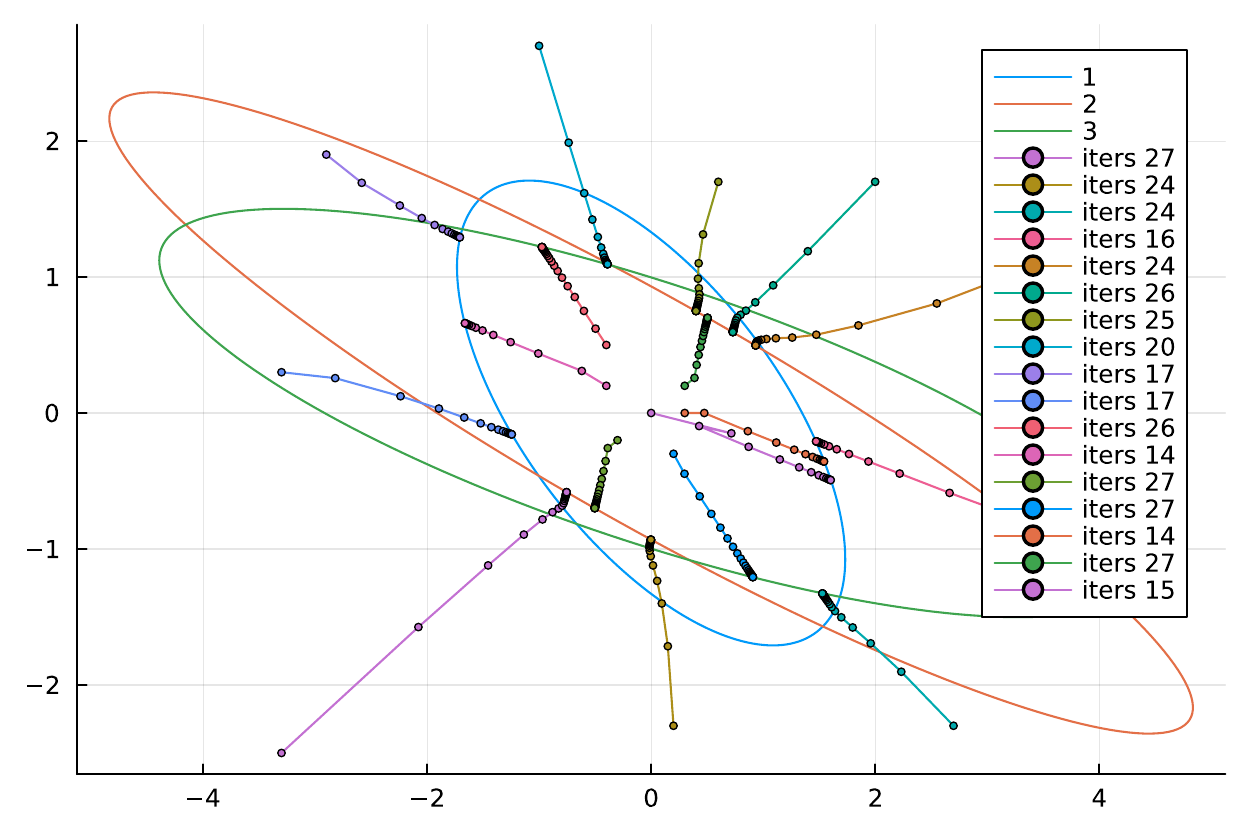}
  \end{center}
 \end{minipage}
 \begin{minipage}{0.5\hsize}
  \begin{center}
   \includegraphics[width=75mm]{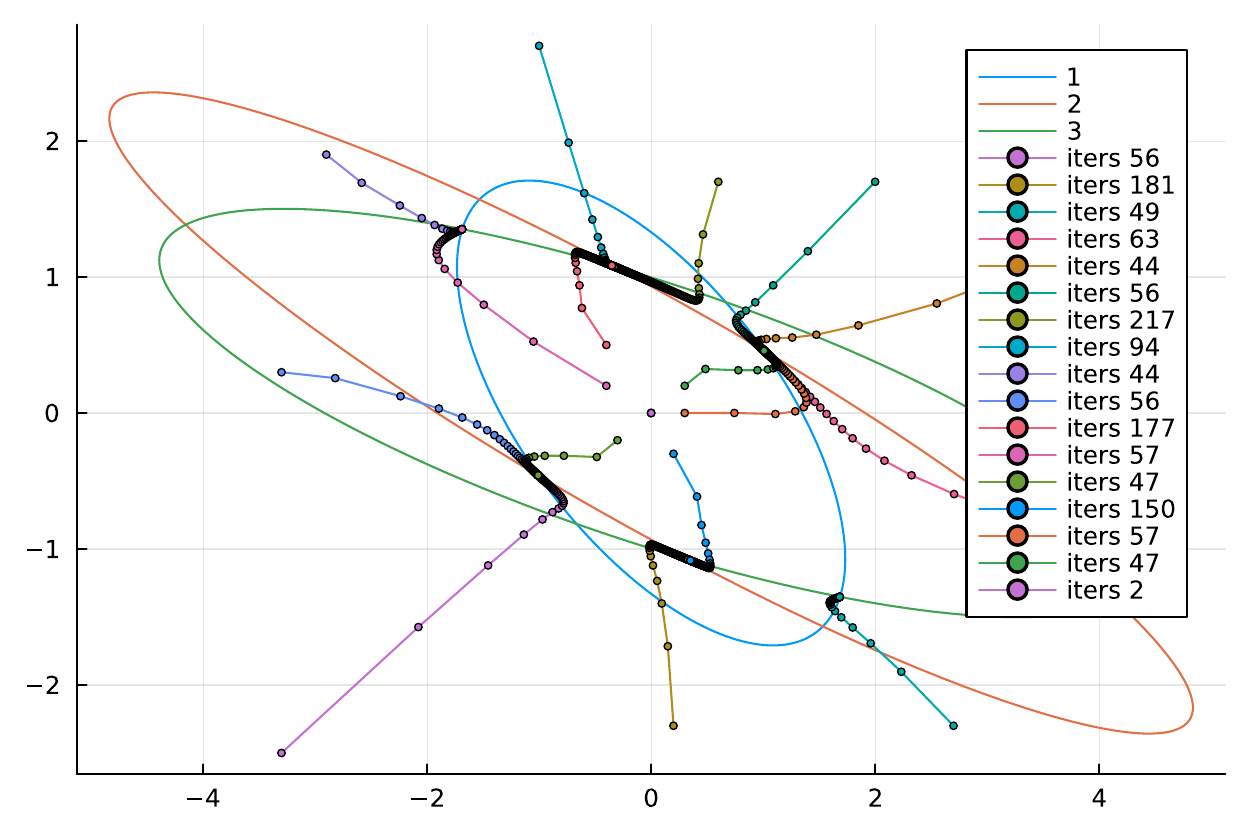}
  \end{center}
 \end{minipage}
  \caption{Illustration of the trajectory $\{x^k\}$ from various initial points for the problem \eqref{eq:ei} with $n=2$ and $m=3$. The left and right figures correspond to the cases for $\ell=1$ and $\ell=2$, respectively. Each line except the ellipsoid is a trajectory of the proposed method where `iters' shows the number of iterations.}
  \label{fig:ellip}
\end{figure}

\section{Concluding remarks}\label{sec:conc}
In this work, using the FTvN system framework developed by Gowda \cite{G19}, we discussed eigenvalue programs beyond the usual matrix setting. 
We also analyzed and implemented a simple projected gradient algorithm, see Algorithm~\ref{alg:project}. 
Finally, we also showed  some applications such as general inverse eigenvalue problems. 
In the particular case where the FTvN system comes from the Euclidean Jordan algebra of second order cones, we showed that it is possible to express the so-called \emph{vanishing quadratic constraints}, which is useful to obtain boundary points of the intersection of ellipsoids with certain degeneracy properties.

The topics we discussed in this paper are still in their infancy and much remains to be done from both modelling 
and algorithmic perspectives. We also mention in passing that \cite{G19} and \cite[Section~4]{JJ23} contain many other examples of 
FTvN systems, including ones corresponding to complete isometric hyperbolic polynomials \cite{BGLS01}  and FTvN systems constructed from infinite-dimensional vector spaces.
One particular direction where many fruitful results may yet be found is related to rank problems involving second-order cones. As problems over second order cones are typically more scalable than matrix problems, fitting a particular modelling application that requires some rank-like constraint into this setting may lead to better performance when compared with a model that uses matrices.

While the writing of this paper was in the final stages, Garner, Lerman and Zhang \cite{GLZ23} released a preprint that is very close to the spirit of what we propose in this paper.
There too they consider problems with general constraints  on eigenvalues and analyze methods based on the idea of solving subproblems on the space of eigenvalues.
That said, one important difference between the work \cite{GLZ23} and our approach is that our discussion uses the FTvN framework while \cite{GLZ23} only consider problems over matrices. 
In particular, the applications involving the second-order cone discussed here seem to be out of the scope of \cite{GLZ23}.

\section*{Acknowledgements}
We thank the referees for their  comments, which helped to improve the paper.

\bibliographystyle{abbrvurl}
\bibliography{bib}
\end{document}